\documentclass{amsart}
\usepackage{setspace}
\usepackage{a4}
\usepackage{amssymb,amsmath,amsthm,latexsym}
\usepackage{amsfonts}
\usepackage{amsfonts}
\usepackage{graphicx}
\usepackage{textcomp}
\usepackage{cite}
\usepackage{enumerate}
\usepackage[mathscr]{euscript}
\usepackage{mathtools}
\newtheorem{theorem}{Theorem}[section]

\newtheorem{definition}[theorem]{Definition}

\newtheorem{lemma} [theorem]{Lemma}

\title{This is the title}
\usepackage{amssymb}
\usepackage{amssymb}
\usepackage{amssymb}
\usepackage{amssymb}
\usepackage{amsmath}
\usepackage{tikz}
\usepackage{hyperref}
\usepackage{enumerate}
\usepackage{mathtools}
\usepackage{amsmath}
\usepackage{tikz}
\usepackage{amssymb}
\usepackage{amsmath}
\usepackage{tikz}
\usepackage{hyperref}
\begin{document}
\begin{center}
{\bf{New Identity on Parseval p-Approximate Schauder Frames and Applications}}\\
K. Mahesh Krishna and P. Sam Johnson  \\
Department of Mathematical and Computational Sciences\\ 
National Institute of Technology Karnataka (NITK), Surathkal\\
Mangaluru 575 025, India  \\
Emails: kmaheshak@gmail.com,  sam@nitk.edu.in\\

Date: \today
\end{center}

\hrule
\vspace{0.5cm}
\textbf{Abstract}: A very useful identity for Parseval frames for Hilbert spaces was obtained by Balan, Casazza, Edidin, and Kutyniok. In this paper, we obtain a similar identity for Parseval p-approximate Schauder frames for Banach spaces which admits a homogeneous semi-inner product in the sense of Lumer-Giles.




\section{Introduction}
	Let $\mathcal{H}$ be a separable Hilbert space over $\mathbb{K}$ ($\mathbb{R}$ or $\mathbb{C}$). A sequence $\{\tau_n\}_n$ in  $\mathcal{H}$ is said to be a
	frame for $\mathcal{H}$ if there exist two constants $0<a\leq b <\infty$ for which
	\begin{align}\label{FRAMEINEQUALITY}
	a\|h\|^2 \leq \sum_{n=1}^\infty |\langle h, \tau_n\rangle|^2\leq b\|h\|^2, \quad \forall h \in \mathcal{H}.
	\end{align}
	The constants $a$ and $b$ in (\ref{FRAMEINEQUALITY}) are called as the lower and upper frame bound, respectively. 	The largest lower frame bound and the smallest upper frame bound are called the optimal frame bounds. 	If  $a=1=b$, then the frame is called as a Parseval frame (see \cite{DUFFIN, OLEBOOK}). Let  $\{\tau_n\}_n$ be a frame for  $\mathcal{H}$. Then it is well-known that 
the frame operator 
$
S_\tau :\mathcal{H} \ni h \mapsto \sum_{n=1}^\infty \langle h, \tau_n\rangle\tau_n\in
\mathcal{H}
$
is  a well-defined  bounded linear, positive and invertible operator (see \cite{OLEBOOK}). Further, the sequence $\{\tilde{\tau}_n:= S_\tau^{-1}\tau_n\}_n$ is also a frame for  $\mathcal{H}$ which is known as the canonical dual frame for $\{\tau_n\}_n$. For a given subset $\mathbb{M} $ of $ \mathbb{N}$, we denote the complement of $\mathbb{M} $ by $\mathbb{M}^\text{c}$ and define $S_\mathbb{M} :  \mathcal{H}\to \mathcal{H}$ by $S_\mathbb{M}(h)= \sum_{n\in \mathbb{M}} \langle h, \tau_n\rangle\tau_n$. By the inequalities in (\ref{FRAMEINEQUALITY}), $
S_\mathbb{M} $
 is a well-defined bounded positive operator which may not be invertible.  We have the following identities for frames for Hilbert spaces  given by Balan, Casazza, Edidin, and Kutyniok.
  
\begin{theorem}\cite{BALACASAZZAEDIDINKUTYNIOK, BALACASAZZAEDIDINKUTYNIOKFIRST} (Frame identity)  \label{CASAZZAGENERAL}
	Let $\{\tau_n\}_n$ be a  frame for  $\mathcal{H}$. Then for every $\mathbb{M} \subseteq \mathbb{N}$, 	
	\begin{align*}
	\sum_{n\in \mathbb{M}}|\langle h, \tau_n\rangle|^2-\sum_{n=1}^\infty|\langle S_\mathbb{M}h, \tilde{\tau}_n\rangle|^2=\sum_{n\in \mathbb{M}^\text{c}}|\langle h, \tau_n\rangle|^2-\sum_{n=1}^\infty|\langle S_{\mathbb{M}^\text{c}}h, \tilde{\tau}_n\rangle|^2,\quad \forall h \in \mathcal{H}.
	\end{align*}	
\end{theorem}
\begin{theorem}\cite{BALACASAZZAEDIDINKUTYNIOK, BALACASAZZAEDIDINKUTYNIOKFIRST} (Parseval frame identity) \label{SECOND}
	Let $\{\tau_n\}_n$ be a Parseval frame for  $\mathcal{H}$. Then for every $\mathbb{M} \subseteq \mathbb{N}$, 
	\begin{align*}
	\sum_{n\in \mathbb{M}}|\langle h, \tau_n\rangle|^2-\left\|\sum_{n\in \mathbb{M}}\langle h, \tau_n\rangle \tau_n\right\|^2=\sum_{n\in \mathbb{M}^\text{c}}|\langle h, \tau_n\rangle|^2-\left\|\sum_{n\in \mathbb{M}^\text{c}}\langle h, \tau_n\rangle \tau_n\right\|^2,\quad \forall h \in \mathcal{H}.
	\end{align*}
\end{theorem}
Theorem \ref{SECOND} is applied to get the  following remarkable lower estimate for Parseval frames.
\begin{theorem}\cite{BALACASAZZAEDIDINKUTYNIOK, GAVRUTA}\label{THIRD}
	Let $\{\tau_n\}_n$ be a Parseval frame for  $\mathcal{H}$. Then for every $\mathbb{M} \subseteq \mathbb{N}$, 
	\begin{align*}
	\sum_{n\in \mathbb{M}}|\langle h, \tau_n\rangle|^2+\left\|\sum_{n\in \mathbb{M}^\text{c}}\langle h, \tau_n\rangle \tau_n\right\|^2=\sum_{n\in \mathbb{M}^\text{c}}|\langle h, \tau_n\rangle|^2+\left\|\sum_{n\in \mathbb{M}}\langle h, \tau_n\rangle \tau_n\right\|^2\geq \frac{3}{4}\|h\|^2,\quad \forall h \in \mathcal{H}.
	\end{align*}
	Further, the bound 3/4 is optimal.
\end{theorem}
Theorem \ref{THIRD} is used in the study of Parseval frames with finite excesses (see \cite{BAKICBERIC, BALANCASAZZAHEIL12}). 
In this paper, we obtain Theorems \ref{CASAZZAGENERAL}, \ref{SECOND} and \ref{THIRD} for certain classes of Banach spaces known as homogeneous semi-inner product spaces. These spaces are introduced by Lumer \cite{LUMER} and studied extensively by Giles \cite{GILES} (see \cite{DRAGOMIR} for a comprehensive look on semi-inner products). We now recall the fundamentals of semi-inner products. Let $\mathcal{X}$ be a vector space over $\mathbb{K}$. A map $[\cdot, \cdot]:\mathcal{X}\times \mathcal{X} \to \mathbb{K}$ is said to be a homogeneous semi-inner product if it satisfies the following:
\begin{enumerate}[\upshape(i)]
	\item $[x,x]>0$, for all $x \in \mathcal{X}, x\neq 0$.
	\item $[\lambda x,y]=\lambda [x,y]$, for all $x,y \in \mathcal{X}$, for all $\lambda \in \mathbb{K}$.
\item 	$[x, \lambda y]=\overline{\lambda} [x,y]$, for all $x,y \in \mathcal{X}$, for all $\lambda \in \mathbb{K}$.
	\item $[x+y, z]=[x, z]+[y, z]$, for all $x,y,z \in \mathcal{X}$.
	\item $| [x,y]|^2\leq [x,x][y,y]$, for all $x,y \in \mathcal{X}$.
\end{enumerate}
A homogeneous semi-inner product $[\cdot, \cdot]$ induces a norm which is defined as $\|x\|\coloneqq \sqrt{[x,x]}$. A prototypical example to keep in mind while working on homogeneous semi-inner product spaces is the standard $\ell^p(\mathbb{N})$ space ($1<p<\infty$) equipped with semi-inner product defined as follows : For $x=\{x_n\}_n,$  $y=\{y_n\}_n\in \ell^p(\mathbb{N})$, define 
\begin{align*}
[x,y]\coloneqq \begin{cases*}
\frac{\sum\limits_{n=1}^{\infty}x_n\overline{y_n}|y_n|^{p-2}}{\|y\|_p^{p-2}} \quad& if $ y \neq 0 $ \\
0 & if $ y=0.$ \\
\end{cases*}
\end{align*}
We now see Riesz  representation theorem for certain classes of Banach spaces. 
\begin{theorem}\cite{GILES}\label{RIESZ}
Let $\mathcal{X}$ be a complete homogeneous semi-inner product space. If $\mathcal{X}$ is continuous and uniformly convex, then for every bounded linear functional $f:  \mathcal{X}\to \mathbb{K}$, there exists a unique $y \in \mathcal{X}$  such that $f(x)=[x,y]$, for all $x \in \mathcal{X}$.
\end{theorem}
Theorem \ref{RIESZ} leads to the notion of generalized adjoint  whose existence is assured in the following theorem.
\begin{theorem}\cite{KOEHLER}\label{THEOREMK}
Let $\mathcal{X}$ be a complete homogeneous semi-inner product space. If $\mathcal{X}$ is continuous and uniformly convex, then for every bounded linear operator  $A:  \mathcal{X}\to \mathcal{X}$, there exists a unique map $A^\dagger:\mathcal{X}\to \mathcal{X}$,  which may not be linear or continuous (called as generalized adjoint of $A$)   such that 
\begin{align*}
[Ax,y]=[x,A^\dagger y],\quad \forall x,y \in \mathcal{X}.
\end{align*}
Moreover, the following statements hold.
\begin{enumerate}[\upshape(i)]
	\item $(\lambda A)^\dagger=\overline{\lambda}A^\dagger$, for all $\lambda \in \mathbb{K}$.
	\item $A^\dagger$ is injective if and only if $\overline{ A(\mathcal{X})}=\mathcal{X}$.
	\item If the norm of $\mathcal{X}$ is strongly (Frechet) differentiable, then $A^\dagger$ is continuous. 
 \end{enumerate}	
\end{theorem}
We next recall the notion of frame for Banach spaces. There are seven different notions of frames for Banach spaces namely, Banach frames with respect to BK-spaces  (see \cite{GROCHENIG, CASAZZAHANLARSON}), frames in a Banach space with respect to a model space of sequences (see \cite{TEREKHIN,  TEREKHIN2}), projection frames (see \cite{TEREKHIN}), p-frames (see \cite{CHRISTENSENPSTOEVA, ALDROUBI}),  $\mathcal{X}_d$-frames (see \cite{CASAZZACHRISTENSENSTOEVA, STOEVA}), cb-frames (see \cite{LIURUAN}) and  Schauder frames or framing (see \cite{CASAZZAHANLARSON, HANLARSONLIU}).   Schauder frames are more flexible and are generalized to approximate Schauder frames in \cite{FREEMANODELL, THOMAS}.
\begin{definition}\cite{FREEMANODELL, THOMAS}\label{ASFDEF}
	Let $\{\tau_n\}_n$ be a sequence in  $\mathcal{X}$ and 	$\{f_n\}_n$ be a sequence in  $\mathcal{X}^*.$ The pair $ (\{f_n \}_{n}, \{\tau_n \}_{n}) $ is said to be an approximate Schauder frame (ASF) for $\mathcal{X}$ if 
	\begin{align*}
		S_{f, \tau}:\mathcal{X}\ni x \mapsto S_{f, \tau}x\coloneqq \sum_{n=1}^\infty
	f_n(x)\tau_n \in
	\mathcal{X}
	\end{align*}
	is a well-defined bounded linear, invertible operator.
\end{definition} 
A stronger form of Definition \ref{ASFDEF} allowing to switch between $\mathcal{X}$ and $\ell^p(\mathbb{N})$ has been studied recently in \cite{MAHESHJOHNSON} and the definition is  given below. 
  \begin{definition}\cite{MAHESHJOHNSON}\label{PASFDEF}
  An ASF $ (\{f_n \}_{n}, \{\tau_n \}_{n}) $  for $\mathcal{X}$	is said to be p-ASF, $p \in [1, \infty)$ if both the maps 
  \begin{align}\label{ALIGNP}
  &\theta_f: \mathcal{X}\ni x \mapsto \theta_f x\coloneqq \{f_n(x)\}_n \in \ell^p(\mathbb{N}) \text{ and } \\
  &\theta_\tau : \ell^p(\mathbb{N}) \ni \{a_n\}_n \mapsto \theta_\tau \{a_n\}_n\coloneqq \sum_{n=1}^\infty a_n\tau_n \in \mathcal{X}\label{AAA}
  \end{align}
  are well-defined bounded linear operators. A p-ASF  is said to be a  Parseval p-ASF if  $	S_{f, \tau}=I_ \mathcal{X}$, the identity operator on $\mathcal{X}$.
  \end{definition}
  Note that, in terms of inequalities, (\ref{ALIGNP}) and (\ref{AAA}) say that there exist constants $c,d>0$, such that 
  \begin{align}\label{P1}
  &\left(\sum_{n=1}^\infty
  |f_n(x)|^p\right)^\frac{1}{p}\leq c \|x\|, \quad \forall x \in \mathcal{X} \text{ and } \\
  &\left\|\sum_{n=1}^\infty a_n\tau_n\right\|\leq d \left(\sum_{n=1}^\infty
  |a_n|^p\right)^\frac{1}{p}, \quad \forall \{a_n\}_n  \in \ell^p(\mathbb{N}).\label{P2}
  \end{align}
  
 \section{New identity for frames for Banach spaces}\label{SECTIONTWO}
   Throughout this paper we assume that  $\mathcal{X}$ is a continuous, uniformly convex, homogeneous semi-inner product space. Let $ (\{f_n \}_{n}, \{\tau_n \}_{n}) $ be a frame  for $\mathcal{X}$. Theorem \ref{RIESZ} says that each $f_n$ can be identified with unique $\omega_n \in \mathcal{X}$  satisfying $f_n(x)=[x,\omega_n]$, for all $x \in \mathcal{X}$. Note that 
   \begin{align*}
   \sum_{n=1}^{\infty}[x, (S_{\omega, \tau}^{-1})^\dagger\omega_n]S_{\omega, \tau}^{-1}\tau_n=S_{\omega, \tau}^{-1}\left(\sum_{n=1}^{\infty}[ S_{\omega, \tau}^{-1}x, \omega_n]\tau_n\right)=S_{\omega, \tau}^{-1}x, \quad \forall x \in \mathcal{X}.
   \end{align*}
   Hence $ (\{\tilde{\omega}_n\coloneqq(S_{\omega, \tau}^{-1})^\dagger\omega_n \}_{n}, \{\tilde{\tau}_n\coloneqq S_{\omega, \tau}^{-1}\tau_n \}_{n}) $ is a p-ASF for $\mathcal{X}$ which is called as the canonical dual frame for $ (\{\omega_n \}_{n}, \{\tau_n \}_{n}) $. 
    Proposition 2.2 in  \cite{BALACASAZZAEDIDINKUTYNIOK} tells that if operators $U, V:\mathcal{H} \to \mathcal{H}$ satisfy $U+V=I_\mathcal{H}$, then $U-V=U^2-V^2$. The result remains valid for Banach spaces as shown in the following lemma. 
    \begin{lemma}\label{LEMMA}
   	If operators $U, V:\mathcal{X} \to \mathcal{X}$ satisfy $U+V=I_\mathcal{X}$, then $U-V=U^2-V^2$.
   \end{lemma}
\begin{proof}
	We 	imitate the proof of Proposition 2.2 in \cite{BALACASAZZAEDIDINKUTYNIOK}: 
	\begin{align*}
	U-V=U-(I_\mathcal{X}-U)=2U-I_\mathcal{X}=U^2-(I_\mathcal{X}-2U+U^2)=U^2-(I_\mathcal{X}-U)^2=U^2-V^2.
	\end{align*}
\end{proof}
 We are now ready to prove Banach space version of the frame identity (Theorem \ref{CASAZZAGENERAL}).
  \begin{theorem}\label{OURGENERAL}
  Let $ (\{\omega_n \}_{n}, \{\tau_n \}_{n}) $  be a p-ASF for $\mathcal{X}$. 	Then for every $\mathbb{M} \subseteq \mathbb{N}$, 	
  	 \begin{align*}
  	\sum_{n\in \mathbb{M}}[x, \omega_n][\tau_n,x]-\sum_{n=1}^{\infty}[S_\mathbb{M}x,\tilde{\omega}_n][\tilde{\tau}_n,S_\mathbb{M}^\dagger x]=\sum_{n\in \mathbb{M}^\text{c}}[x, \omega_n][\tau_n,x]-\sum_{n=1}^{\infty}[S_{\mathbb{M}^\text{c}}x,\tilde{\omega}_n][\tilde{\tau}_n,S_{\mathbb{M}^\text{c}}^\dagger x] , \quad \forall x \in \mathcal{X}.
  	\end{align*}
  \end{theorem}
  \begin{proof}
 For notational convenience, we denote $S_{f, \tau}$ by $S$. We have $S_\mathbb{M}+S_{\mathbb{M}^\text{c}}=S$. Using 	$S^{-1}S_\mathbb{M}+S^{-1}S_{\mathbb{M}^\text{c}}=I_\mathcal{X}$ and Lemma \ref{LEMMA} we get $S^{-1}S_\mathbb{M}-S^{-1}S_{\mathbb{M}^\text{c}}=(S^{-1}S_\mathbb{M})^2-(S^{-1}S_{\mathbb{M}^\text{c}})^2=S^{-1}S_\mathbb{M}S^{-1}S_\mathbb{M}-S^{-1}S_{\mathbb{M}^\text{c}}S^{-1}S_{\mathbb{M}^\text{c}}$ which gives 
 \begin{align*}
 S^{-1}S_\mathbb{M}-S^{-1}S_\mathbb{M}S^{-1}S_\mathbb{M}=S^{-1}S_{\mathbb{M}^\text{c}}-S^{-1}S_{\mathbb{M}^\text{c}}S^{-1}S_{\mathbb{M}^\text{c}}.
 \end{align*}
 Therefore for all $x, y \in \mathcal{X}$,
  \begin{align*}
 [S^{-1}S_\mathbb{M}x,y]-[S^{-1}S_\mathbb{M}S^{-1}S_\mathbb{M}x, y]=[S^{-1}S_{\mathbb{M}^\text{c}}x,y]-[S^{-1}S_{\mathbb{M}^\text{c}}S^{-1}S_{\mathbb{M}^\text{c}}x,y].
 \end{align*}
 In particular, 
 \begin{align*}
 [S^{-1}S_\mathbb{M}x,S^\dagger x]-[S^{-1}S_\mathbb{M}S^{-1}S_\mathbb{M}x, S^\dagger x]=[S^{-1}S_{\mathbb{M}^\text{c}}x,S^\dagger x]-[S^{-1}S_{\mathbb{M}^\text{c}}S^{-1}S_{\mathbb{M}^\text{c}}x,S^\dagger x], \quad \forall x \in \mathcal{X}
 \end{align*}
 which gives 
 \begin{align}\label{FIRST}
 [S_\mathbb{M}x,x]-[S^{-1}S_\mathbb{M}x, S_\mathbb{M}^\dagger x]=[S_{\mathbb{M}^\text{c}}x,x]-[S^{-1}S_{\mathbb{M}^\text{c}}x,S_{\mathbb{M}^\text{c}}^\dagger x], \quad \forall x \in \mathcal{X}.
 \end{align}
 Note that 
 \begin{align*}
 \sum_{n=1}^{\infty}[x,\tilde{\omega}_n][\tilde{\tau}_n,y]&=\sum_{n=1}^{\infty}[x,(S^{-1})^\dagger\omega_n][S^{-1}\tau_n,y]=\sum_{n=1}^{\infty}[S^{-1}x, \omega_n][S^{-1}\tau_n,y]\\
 &=\left[\sum_{n=1}^{\infty}[S^{-1}x, \omega_n]S^{-1}\tau_n,y \right]=\left[S^{-1}\left(\sum_{n=1}^{\infty}[S^{-1}x, \omega_n]\tau_n\right),y \right]=[S^{-1}x,y], \quad \forall x,y \in \mathcal{X}.
 \end{align*}
 Equation (\ref{FIRST}) now gives 
 \begin{align*}
\sum_{n\in \mathbb{M}}[x, \omega_n][\tau_n,x]-\sum_{n=1}^{\infty}[S_\mathbb{M}x,\tilde{\omega}_n][\tilde{\tau}_n,S_\mathbb{M}^\dagger x]=\sum_{n\in \mathbb{M}^\text{c}}[x, \omega_n][\tau_n,x]-\sum_{n=1}^{\infty}[S_{\mathbb{M}^\text{c}}x,\tilde{\omega}_n][\tilde{\tau}_n,S_{\mathbb{M}^\text{c}}^\dagger x] , \quad \forall x \in \mathcal{X}.
 \end{align*}
  \end{proof}

In the next theorem, we prove Banach space version of the Parseval frame identity (Theorem \ref{SECOND}).
\begin{theorem} (Parseval p-ASF identity) \label{PASFID}
  Let $ (\{\omega_n \}_{n}, \{\tau_n \}_{n}) $  be a Parseval p-ASF for $\mathcal{X}$. 	Then for every $\mathbb{M} \subseteq \mathbb{N}$,
  \begin{align*}
  \sum_{n\in \mathbb{M}}[x, \omega_n][\tau_n,x]-\sum_{n\in \mathbb{M}}\sum_{k\in \mathbb{M}}[x, \omega_n][\tau_n,\omega_k][\tau_k,x]=\sum_{n\in \mathbb{M}^\text{c}}[x, \omega_n][\tau_n,x]-\sum_{n\in \mathbb{M}^\text{c}}\sum_{k\in \mathbb{M}^\text{c}}[x, \omega_n][\tau_n,\omega_k][\tau_k,x], \quad \forall x \in \mathcal{X}.
  \end{align*}	
  \end{theorem}
  \begin{proof}
 Using  	Theorem \ref{OURGENERAL}, for all $ x \in \mathcal{X}$,
 
 \begin{align*}
 &\sum_{n\in \mathbb{M}}[x, \omega_n][\tau_n,x]-\sum_{n\in \mathbb{M}}\sum_{k\in \mathbb{M}}[x, \omega_n][\tau_n,\omega_k][\tau_k,x]= \sum_{n\in \mathbb{M}}[x, \omega_n][\tau_n,x]-\left[\sum_{n\in \mathbb{M}}[x, \omega_n]\sum_{k\in \mathbb{M}}[\tau_n,\omega_k]\tau_k, x\right]\\
 &=\sum_{n\in \mathbb{M}}[x, \omega_n][\tau_n,x]-\left[\sum_{n\in \mathbb{M}}[x, \omega_n]S_\mathbb{M}\tau_n, x\right]=\sum_{n\in \mathbb{M}}[x, \omega_n][\tau_n,x]-\left[S_\mathbb{M}\left(\sum_{n\in \mathbb{M}}[x, \omega_n]\tau_n\right), x\right]\\
 &=\sum_{n\in \mathbb{M}}[x, \omega_n][\tau_n,x]-\left[S_\mathbb{M} S_\mathbb{M}x, x\right]=\sum_{n\in \mathbb{M}}[x, \omega_n][\tau_n,x]-\left[ S_\mathbb{M}x, S_\mathbb{M}^\dagger x\right]\\
 &=\sum_{n\in \mathbb{M}}[x, \omega_n][\tau_n,x]-\left[\sum_{n=1}^{\infty}[ S_\mathbb{M}x, \omega_n]\tau_n,S_\mathbb{M}^\dagger x\right]=\sum_{n\in \mathbb{M}}[x, \omega_n][\tau_n,x]-\sum_{n=1}^{\infty}[ S_\mathbb{M}x, \omega_n][\tau_n,S_\mathbb{M}^\dagger x]\\
 &=\sum_{n\in \mathbb{M}^\text{c}}[x, \omega_n][\tau_n,x]-\sum_{n=1}^{\infty}[S_{\mathbb{M}^\text{c}}x,\omega_n][\tau_n,S_{\mathbb{M}^\text{c}}^\dagger x]
 =\sum_{n\in \mathbb{M}^\text{c}}[x, \omega_n][\tau_n,x]-\sum_{n\in \mathbb{M}^\text{c}}\sum_{k\in \mathbb{M}^\text{c}}[x, \omega_n][\tau_n,\omega_k][\tau_k,x].
 \end{align*}
  \end{proof}
In terms of $S_\mathbb{M}$ and $S_\mathbb{M}^\text{c}$, Theorem \ref{PASFID} can be written as 
\begin{align}\label{OPERATORDESCRPTION}
S_\mathbb{M}-S_\mathbb{M}^2=S_{\mathbb{M}^\text{c}}-S_{\mathbb{M}^\text{c}}^2\quad \text{ or } \quad S_\mathbb{M}+S_{\mathbb{M}^\text{c}}^2=S_{\mathbb{M}^\text{c}}+S_\mathbb{M}^2.
\end{align}
 We now give an application of Theorem \ref{PASFID} which is Banach space version of Theorem \ref{THIRD}. 
  \begin{theorem}\label{LOWER234}
  Let $ (\{\omega_n \}_{n}, \{\tau_n \}_{n}) $  be a Parseval p-ASF for $\mathcal{X}$. 	Let  $\mathbb{M} \subseteq \mathbb{N}$. If $x$ is in $\mathcal{X}$ such that $[(S_\mathbb{M}-\frac{1}{2}I_\mathcal{X})^2x, x]\geq 0$, then 
  \begin{align*}
 \sum_{n\in \mathbb{M}}[x, \omega_n][\tau_n,x]+\sum_{n\in \mathbb{M}^\text{c}}\sum_{k\in \mathbb{M}^\text{c}}[x, \omega_n][\tau_n,\omega_k][\tau_k,x]&=\sum_{n\in \mathbb{M}^\text{c}}[x, \omega_n][\tau_n,x]+\sum_{n\in \mathbb{M}}\sum_{k\in \mathbb{M}}[x, \omega_n][\tau_n,\omega_k][\tau_k,x]\\
 &\geq \frac{3}{4}\|x\|^2 , \quad \text{ for those } x.
  \end{align*}
  \end{theorem}
  \begin{proof}
  	We first compute 
\begin{align*}
 S_\mathbb{M}^2+S_{\mathbb{M}^\text{c}}^2 = S_\mathbb{M}^2+ (I_\mathcal{X}-S_\mathbb{M})^2=2 S_\mathbb{M}^2-2S_\mathbb{M}+I_\mathcal{X}
=2\left(S_\mathbb{M}-\frac{1}{2}I_\mathcal{X}\right)^2+\frac{1}{2}I_\mathcal{X}.
\end{align*} 
Hence if $x$ is in $\mathcal{X}$ satisfying $[(S_\mathbb{M}-\frac{1}{2}I_\mathcal{X})^2x, x]\geq 0$, then 
\begin{align*}
[( S_\mathbb{M}^2+S_{\mathbb{M}^\text{c}}^2)x,x]\geq \frac{1}{2}\|x\|^2.
\end{align*}
Now using Equation (\ref{OPERATORDESCRPTION}) we get 
\begin{align*}
&2\sum_{n\in \mathbb{M}}[x, \omega_n][\tau_n,x]+2\sum_{n\in \mathbb{M}^\text{c}}\sum_{k\in \mathbb{M}^\text{c}}[x, \omega_n][\tau_n,\omega_k][\tau_k,x]=2[S_\mathbb{M}x,x]+2[ S_{\mathbb{M}^\text{c}}^2x,x]\\
&=[2(S_\mathbb{M}+S_{\mathbb{M}^\text{c}}^2)x,x]=[((S_\mathbb{M}+S_{\mathbb{M}^\text{c}}^2)+(S_\mathbb{M}+S_{\mathbb{M}^\text{c}}^2))x,x]\\
&=[((S_\mathbb{M}+S_{\mathbb{M}^\text{c}}^2)+(S_{\mathbb{M}^\text{c}}+S_\mathbb{M}^2))x,x]=[(I_\mathcal{X}+S_{\mathbb{M}^\text{c}}^2+S_\mathbb{M}^2)x,x]\\
&=\|x\|^2 +[( S_\mathbb{M}^2+S_{\mathbb{M}^\text{c}}^2)x,x] \geq \frac{3}{4}\|x\|^2 , \quad \forall x \in \mathcal{X}.
\end{align*}
  \end{proof}
Note that even for Hilbert spaces, the bound 3/4 cannot be improved, see Proposition 2.4 in \cite{GAVRUTA}.
  \section{Acknowledgements}
   First author thanks National Institute of Technology Karnataka (NITK) Surathkal for financial assistance.

 \bibliographystyle{plain}
 \bibliography{reference.bib}

\end{document}